\documentclass[12pt]{amsart}
\usepackage{amsfonts}
\usepackage{bbm}
\usepackage{amsfonts,amssymb,amsmath,amsthm}
\usepackage{url}
\usepackage{enumerate}
\usepackage{bbm}
\usepackage{amsfonts}
\usepackage{mathrsfs}
\usepackage{amsmath,amscd,amsthm,amsfonts}
\usepackage[all]{xy}
\usepackage[latin1]{inputenc}
\usepackage[pdftex, colorlinks, citecolor=red, backref=page]{hyperref}

\newtheorem{thm}{Theorem}[section]
\newtheorem{defi}[thm]{Definition}
\newtheorem{lem}[thm]{Lemma}
\newtheorem{coro}[thm]{Corollary}

\newtheorem{prop}[thm]{Proposition}
\newtheorem{rem}[thm]{Remark}

\theoremstyle{definition}
\numberwithin{equation}{section}

\newcommand{\Z}{Z\cap \overline{W^u(x, \delta)}}
\newcommand{\Sc}{S^u(f, \epsilon, n, \eta(x), \gamma)}

\author {Wenda Zhang}
\address{College of Mathematics and Statistics, Chongqing Jiaotong University, Chongqing, China \ 400074}
\email{wendazhang951@aliyun.com}
\author{zhiqiang li}
\address{College of Mathematics and Statistics, Chongqing University, Chongqing, China \ 401331}
\email{zqli@cqu.edu.cn}
\author{yunhua zhou}
\address{College of Mathematics and Statistics, Chongqing University, Chongqing, China \ 401331}
\email{zhouyh@cqu.edu.cn}

\begin{document}

\title[unstable metric pressure for sub-additive potentials]{unstable metric pressure of partially hyperbolic diffeomorphisms with sub-additive potentials}

\begin{abstract}
In this paper, we define and study unstable measure theoretic pressure for $C^1$-smooth partially hyperbolic diffeomorphisms with sub-additive potentials. We show that this measure theoretic pressure for any ergodic measure equals the corresponding unstable measure theoretic entropy plus the \emph{Lyapunov exponents}\ of the potentials with respect to the ergodic measure. On the other hand, we also give other definitions of unstable metric pressure, in terms of the Bowen's picture and the capacity picture. We show that all definitions of unstable metric pressure, including the one defined at the beginning, actually coincide for any ergodic measure.\end{abstract}

\maketitle

\section{Introduction}As a natural generalization of topological entropy, topological pressure for a given continuous function on the phase space roughly measures the orbit complexity of iterated maps on the potential functions. In \cite{Ruelle}, Ruelle first defined topological pressure for expansive maps. Under some assumptions, he also established a variational principle, which was generalized by in \cite{Walters1} by Walters in full generality. In \cite{Pesin1}, Pesin and Pitskel defined topological
pressures for non-compact subsets and proved a variational principle under some supplementary conditions. Based on Katok's work \cite{Kat}, He, Lv, and Zhou \cite{He} introduced measure theoretic pressure for ergodic measures. All pressure mentioned are about additive potentials--the sequence of continuous functions consisting of summations over orbits of the dynamical map.     

On the other hand, sub-additive potentials for a dynamical system is a sequence of continuous functions satisfying sub-additivity condition involving the dynamical map. In \cite{K. Falconer}, Falconer first introduced topological pressures for sub-additive potentials on mixing repellers. Barreira in \cite{Bar} generalized Pesin and Pitskel's work \cite{Pesin1} to topological pressure for general potentials. With restrictive assumptions on the potentials, they proved variational principles. In \cite{Huang1}, without any restrictions, Cao, Feng, and Huang obtained a variational principle of topological pressure for sub-additive potentials. Furthermore, Cheng, Cao, Hu, and Zhao investigated measure theoretic pressure for non-additive potentials, see \cite{Cheng}, \cite{Hu}. 

In recent years, the theory of entropy and pressure for $C^1$-smooth partially hyperbolic diffeomorphisms are
intensively investigated. In \cite{Hu1}, Hu, Hua, and Wu introduced the unstable topological and metric entropy, obtained the corresponding Shannon-McMillan-Breiman theorem, local entropy formula, and established the corresponding variational principle. The main feature of these unstable entropies is to
rule out the complexity on central directions and focus on that on unstable directions. In fact, the unstable metric entropy in  \cite{Hu1} has root in the entropy introduced by Ledrappier
and Young (\cite{Ledrappier3}), and is easier to apply. 
In \cite{Tian}, Tian and Wu generalize the above result with additional consideration of an arbitrary subset (not necessarily compact or invariant).
In \cite{Zhu2}, Hu, Wu, and Zhu investigated the unstable topological pressure for additive potentials, and obtained a variational principle. 

It is a natural task to extend pressure theory to the case of sub-additive potentials of $C^1$-smooth partially hyperbolic diffeomorphisms. In \cite{Zhang}, we introduce sub-additive unstable topological pressure, and set up a corresponding variational principle. 

In this paper, we define and study sub-additive unstable measure theoretic pressure. For any ergodic measure, we show that this metric pressure equals the corresponding unstable metric entropy plus the corresponding \emph{Lyapunov exponents} with respect to the measure. Moreover, we also formulate and study other definitions of unstable metric pressure, in terms of the Bowen's picture and the capacity picture. It turns out that all definitions of unstable metric pressure, including the one defined at the beginning, actually coincide for any ergodic measure.

Our main results read as follows.
\begin{thm}\label{variational principle} 
Let there be given a $C^{1}$-smooth partially hyperbolic diffeomorphism $f: M\rightarrow M$, and a sequence of sub-additive potentials $\mathcal{G}=\{\log g_n\}_{n\geq 1}$ of $f$ on $M$. Then for any  $\mu\in\mathcal{M}^{e}_{f}(M)$, we have
$$
P_{\mu}^{u}(f,\mathcal{G})=h^{u}_{\mu}(f)+\mathcal{G}_{*}(\mu).
$$
\end{thm}
Combing with Theorem 1.1 in \cite{Zhang}, we have the following variational principle.
\begin{coro}\label{varprin}
Let $f : M \rightarrow M$ be a $C^{1}$ partially hyperbolic diffeomorphism and $\mathcal{G} = \{\log g_{n}\}
_{n=1}^{\infty}$ be a sequence of sub-additive potentials of $f$ on $M$. Then
$$
P^{u}(f,\mathcal{G})=\sup \{P_{\mu}^{u}(f,\mathcal{G}): \mu\in\mathcal{M}^{e}_{f}(M)\}.
$$
\end{coro}

\begin{thm}\label{Main Result}
Let there be given a $C^{1}$-smooth partially hyperbolic diffeomorphism $f: M\rightarrow M$, and a sequence of sub-additive potentials $\mathcal{G}=\{\log g_n\}_{n\geq 1}$ of $f$ on $M$. Then for any $\mu\in \mathcal{M}^e_{f}(M)$, one has
$$P^{u}_{\mu}(f, \mathcal{G})=\underline{CP}^u_\mu(f, \mathcal{G})=\overline{CP}^u_\mu(f, \mathcal{G})=P^u_{B, \mu}(f, \mathcal{G}).$$
\end{thm}
(All terms involved are defined in Section 2 and 4, see in particular Definition \ref{metric-pressure}, \ref{Bowen-Top-Pressure}, \ref{Bowen-Metric-Pressure}, \ref{Capacity-Top-Pressure}, and \ref{Capacity-Metric-Pressure}. The sets $\mathcal{M}_{f}(M)$ and $\mathcal{M}^{e}_{f}(M)$ refer to the collection of $f$-invariant and ergodic probability
measures on $M$ respectively.)

The paper is organized as follows. In Section 2, we set up notation, and give definition of the unstable measure theoretic pressure for sub-additive potentials. In Section 3, we prove Theorem \ref{variational principle} in two steps. In Section 4, we give other definitions of unstable metric pressure, in terms of the Bowen's picture and the capacity picture. Moreover, we give a proof of Theorem \ref{Main Result}.

\section{Notation and definitions.}

Let $M$ be an $n$-dimensional, smooth, connected, and compact Riemannian manifold
without boundary; and $f : M \rightarrow M$ be a $C^{1}$-diffeomorphism. We say $f$ is 
\emph{partially hyperbolic}, if there exists a nontrivial $Df$-invariant
splitting $TM=E^{s}\bigoplus E^{c}\bigoplus E^{u}$ of the tangent bundle into stable, central, and unstable
distributions, such that all unit vectors $v^{\sigma}\in E^{\sigma}_{x} (\sigma=s, c, u)$ with $x \in M$ satisfy
$$\parallel D_{x}fv^{s} \parallel< \parallel D_{x}fv^{c} \parallel<\parallel D_{x}fv^{u} \parallel,$$
and
$$\parallel D_{x}f\mid_{E^{s}_{x}}\parallel< 1 \;\;\mbox{and} \;\;\parallel D_{x}f^{-1}\mid_{E^{u}_{x}}\parallel< 1,$$
for some suitable Riemannian metric on $M$. The stable distribution $E^{s}$ and unstable
distribution $E^{u}$ are integrable to the stable and unstable foliations $W^{s}$ and $W^{u}$
respectively such that $TW^{s} = E^{s}$ and $TW^{u} = E^{u}$ (cf. \cite{Hirsch}).

In this paper, we always work in the setting of $C^1$-smooth partially hyperbolic system $(M, f).$ 

\begin{defi}
Given a sequence of continuous functions $\mathcal{G} = \{\log g_{n}\}
_{n=1}^{\infty}$ on $M$, $\mathcal{G}$ is
called a sequence of sub-additive potentials of $f$ if 
$$\log g_{m+n}(x)\leq \log g_{n}(x)+\log g_{m}(f^{n}x), \forall x\in M, \forall m,n\in\mathbb{N}.$$

\end{defi}

\begin{rem}
For any $f$-invariant Borel probability measure $\mu$, set
$$\mathcal{G}_{*}(\mu)=\lim\limits_{n\rightarrow \infty}\frac{1}{n}\int\log g_{n}d\mu,$$
and $\mathcal{G}_{*}(\mu)$ is called the \emph{Lyapunov exponent} of $\mathcal{G}$ with respect to $\mu$. The existence of this limit follows from a sub-additive argument. It takes values in
$[-\infty, +\infty)$. Moreover, the Sub-additive Ergodic Theorem (see \cite{Walters2}, Theorem 10.1) implies that for an ergodic measure $\mu$, one has $$\mathcal{G}_{*}(\mu)=\lim\limits_{n\rightarrow \infty}\frac{1}{n}\log g_{n}(x), \mu-a.e.\,x.$$
\end{rem}

Next we recall some basic facts about unstable entropy (see \cite{Hu1}). Given any probability measure $\nu$ and any finite measurable partition $\eta$ of $M$, and denote by $\eta(x)$ the element of $\eta$ containing $x$.
The \emph{canonical system of conditional measures}
for $\nu$ and $\eta$ is a family of probability measures $\{\nu^{\eta}_{x} : x \in M\}$ with $\nu^{\eta}_{x}(\eta(x))= 1$,
such that for every measurable set $B \subset M, x \mapsto \nu^{\eta}_{x}(B)$ is measurable and
$$\nu(B) =\int_{X}\nu^{\eta}_{x}(B)d\nu(x).$$ A classical result of Rokhlin (cf.\,\cite{Rohlin}) says that if $\eta$ is a measurable partition,
then there exists a system of conditional measures with respect to $\eta$. It is essentially
unique in the sense that two such systems coincide for sets with full $\nu$-measure. For
measurable partitions $\alpha$ and $\eta$, let
$$H_{\nu}(\alpha|\eta):=-\int_{M}\log\nu^{\eta}_{x}(\alpha(x))d\nu(x).$$
denote the conditional entropy of $\alpha$ for given $\eta$ with respect to $\nu$.

Take $\epsilon_{0} > 0$
small. Let $\mathcal{P} = \mathcal{P}_{\epsilon_{0}}$ denote the set of finite Borel partitions $\alpha$ of $M$ whose elements
have diameters smaller than or equal to $\epsilon_{0}$, that is, $\mbox{diam}\;\alpha := \sup\{\mbox{diam}\;A : A \in
\alpha\} \leq \epsilon_{0}$. For each $\beta \in \mathcal{P}$ we can define a finer partition $\eta$  such that $\eta(x) =\beta(x) \cap W^{u}_{loc}(x)$ for each $x \in M$, where $W^{u}_{loc}(x)$ denotes the local unstable manifold
at $x$ whose size is greater than the diameter $\epsilon_{0}$ of $\beta$. Since $W^{u}$ is a continuous
foliation, $\eta$ is a measurable partition with respect to any Borel probability measure
on $M$. 

Let $\mathcal{P}^{u}$ denote the set of partitions $\eta$ obtained in this way and \emph{subordinate
to unstable manifolds}. Here a partition $\eta$ of $M$ is said to be subordinate to unstable
manifolds of $f$ with respect to a measure $\mu$ if for $\mu$-almost every $x, \eta(x) \subset W^{u}(x)$
and contains an open neighborhood of $x$ in $W^{u}(x)$. It is clear that if $\alpha \in \mathcal{P}$
satisfies $\mu(\partial\alpha) = 0$, where $\partial\alpha := \cup_{A\in \alpha}\partial A$, then the corresponding $\eta$ given by
$\eta(x) = \alpha(x) \cap W^{u}_{loc}(x)$ is a partition subordinate to unstable manifolds of $f$.

The unstable metric entropy in \cite{Hu1} is defined as follows.
\begin{defi} For any $\mu\in\mathcal{M}_f(M)$, any $\eta\in \mathcal{P}^u$, and any $\alpha\in\mathcal{P}$, define $$h_{\mu}(f, \alpha|\eta)=\limsup_{n\rightarrow\infty}\frac{1}{n}H_{\mu}(\alpha_{0}^{n-1}|\eta),$$ and $$h_{\mu}(f|\eta)=\sup\limits_{\alpha\in\mathcal{P}}h_{\mu}(f, \alpha|\eta).$$
The unstable metric entropy of $f$ is defined by
$$h^{u}_{\mu}(f)=\sup\limits_{\eta\in\mathcal{P}^{u}}h_{\mu}(f|\eta).$$

\end{defi}

We define unstable metric pressure for sub-additive potentials as follows.

Take any $\eta\in \mathcal{P}^{u}$. A subset $F \subseteq \overline{\eta(x)}$ is called an $(n,\epsilon,\gamma)\; u$-spanning set of $\eta(x)$, if 
$$\mu^{\eta}_{x} \left(\cup_{y\in F}B_{n}^{u}(y,\epsilon)\right)\geq 1-\gamma,$$
where $B^{u}_{n}(y,\epsilon) = \{z \in W^{u}(x) : d^{u}_{n}(y, z) \leq\epsilon\}$ is the $(n, \epsilon)\; u$-Bowen
ball around $y$. 

\begin{defi}\label{metric-pressure}
 For any $x\in M$, any $\eta\in \mathcal{P}^u$, any positive number $\gamma$, any natural number $n$, any sequence $\mathcal{G}$ of sub-additive potentials of $f$ on $M$, and any $\mu \in \mathcal{M}_{f}(M)$, set
$$\begin{aligned}
&P_{\mu}^{u}(f, \mathcal{G}, \epsilon, n, \eta(x), \gamma)\\
:=&\inf\left\{\sum_{y\in F}\sup_{z\in B_{n}^{u}(y,\epsilon)}g_{n}(z)\mid F\;\text{is an} \;(n, \epsilon, \gamma)\, u\text {-spanning subset of } \eta(x)\right\},
\end{aligned}$$
$$P_{\mu}^{u}\left(f, \mathcal{G}, \epsilon,\eta(x),\gamma\right):=\limsup\limits_{n\rightarrow\infty}\frac{1}{n}\log P_{\mu}^{u}(f, \mathcal{G}, \epsilon, n, \eta(x), \gamma),$$ and
$$
P_{\mu}^{u}\left(f, \mathcal{G},\eta(x), \gamma \right):=\lim\limits_{\epsilon\rightarrow 0}P_{\mu}^{u}(f, \mathcal{G}, \epsilon, \eta(x), \gamma).
$$
The unstable measure-theoretic pressure of $f$ with respect to $\mathcal{G}$ is defined by
$$P_{\mu}^{u}(f, \mathcal{G}):=\sup_{\eta\in\mathcal{P}^{u}}\int_{M}\lim_{\gamma\rightarrow 0}P_{\mu}^{u}\left(f, \mathcal{G}, \eta(x),\gamma \right)d\mu(x).$$
\end{defi}

\begin{rem}
For any continuous function $\varphi\in C(M)$, the corresponding sequence $\mathcal{G}=\{S_n\varphi(x)=\sum\limits_{i=0}^{n-1}\varphi(f^ix)\}$ is additive and hence sub-additive. We simply write $P_{\mu}^{u}(f, \mathcal{G})$ as $P_{\mu}^{u}(f, \varphi)$, which actually coincides with the classical definition.
\end{rem}

\section{unstable metric pressure equals unstable metric entropy plus \textit{Lyapunov exponent}}
In this section, we prove Theorem \ref{variational principle} in two steps. First we show the conclusion is true in the case of additive potentials. Second we prove Theorem \ref{variational principle} for sub-additive potentials, with some help of the previous case.

\subsection{The case of additive potentials.}
\begin{thm}\label{metric pressure formular}
For any $\varphi \in C(M,\mathbb{R})$ and $\mu \in \mathcal{M}^{e}_{f}(M)$, we have
$$P_{\mu}^{u}(f, \varphi)= h^{u}_{\mu}(f)+\int_{M}\varphi d\mu.$$
\end{thm}
The proof of this theorem splits into the following two lemmas.
\begin{lem}\label{metric pressure formular 1}
For any $\varphi \in C(M,\mathbb{R})$ and $\mu \in \mathcal{M}^{e}_{f}(M)$, we have
$$P_{\mu}^{u}(f, \varphi)\leq h^{u}_{\mu}(f)+\int_{M}\varphi d\mu.$$
\end{lem}
\begin{proof}

Given any $\epsilon>0$, any $0<\gamma<1$, any large $n\in\mathbb{N}$, any $\rho>0$, and any $\eta\in \mathcal{P}^u$.
Let us choose a finite partition $\alpha$ of $M$ such that the diameter of $\alpha$ is less than $\epsilon/2C$, where $C>1$ satisfies $$d(y, z)\leq d^u(y, z)\leq Cd(y, z)\;\text{for any}\;y, z\in \eta(x).$$  
Since $\mu$ is ergodic, according to the Theorem B in \cite{Hu1}, one has $$\lim\limits_{n\to\infty}-\frac{1}{n}\log\mu^{\eta}_{x}(\alpha^{n-1}_{0}(x))=h^u_{\mu}(f), \;\;\;\mu-a.e.\,x\in M.$$ Hence for $\mu-a.e. \;x\in M$, one has $$\lim\limits_{n\to\infty}-\frac{1}{n}\log\mu^{\eta}_{x}(\alpha^{n-1}_{0}(y))=h^u_{\mu}(f), \;\;\;\mu^{\eta}_x-a.e.\,y\in\eta(x).$$ Then for $\mu^{\eta}_{x}-a.e.\,y$, there exists an $N(y, \rho)>0$ such that if $n\geq N(y, \rho)$, then $$\mu^{\eta}_{x}(\alpha^{n-1}_{0}(y))\geq e^{-n(h^u_{\mu}(f)+\rho)}.$$ Set $$E_n=\{y\in\eta(x)\mid N(y, \rho)\leq n\},$$ then $$\mu^{\eta}_{x}(\bigcup^{\infty}_{n=1}E_n)=1.$$ So $\mu^{\eta}_x(E_n)\geq1-\gamma/2$ if $n$ is large enough. Then it is easy to see that $E_n$ intersects at most $e^{n(h^u_{\mu}(f)+\rho)}$ members of $\alpha^{n-1}_{0}$ and can be covered by the same number of $(n, \epsilon)\;u$-Bowen balls. If we take a point from each member of $\alpha^{n-1}_0\cap E_n$, then it is clear that they contribute to an $(n, \epsilon)\;u$-spanning set $F_n$ of $E_n$. Moreover, $$|F_n|\leq e^{n(h^u_{\mu}(f)+\rho)}\quad\quad(*).$$

On the other hand, according to the Birkhoff's ergodic theorem, one has
 $$\lim_{n\rightarrow\infty}\dfrac{1}{n}S_{n}\varphi(y)=\int_{M}\varphi d\mu, \;\;\;\mu-a.e.\;y.$$
By the Egoroff's Theorem, there is a measurable set $B$ with $\mu(B)\geq1-\gamma/2$, and $(1/n)S_n\varphi$ converges uniformly to $\int_{M}\varphi d\mu$ on $B$.

So if one can take $n$ to be further large enough, and set $E=B\cap E_n$, then $\mu_x^{\eta}(E)>1-\gamma$; moreover, $$\dfrac{1}{n}S_{n}\varphi(y)\leq\int_{M}\varphi d\mu+\rho,\;\text{for all}\;y\in E.$$

Take $F$ to be an $(n, \epsilon)\, u$-spanning set of $E$ with the smallest cardinality, then $|F|\leq e^{n(h^u_{\mu}(f)+\rho)}$ based on $(*)$. Then for any $z\in F$, there is a $y(z)\in E$ such that $d^u_n(z, y(z))<\epsilon$. Therefore,
\begin{align*}
\sum_{z\in F}\exp(S_{n}\varphi)(z)\leq &\sum_{z\in F}\exp((S_{n}\varphi)(y(z))+n\tau_{\epsilon})\notag\\
\leq& \sum_{z\in F}\exp(n(\int_{M}\varphi d\mu+\rho)+n\tau_{\epsilon})\notag\\
\leq& \exp n(h^u_{\mu}(f)+\int_{M}\varphi d\mu+2\rho+\tau_{\epsilon}),
\end{align*}
where $\tau_{\epsilon}=\sup\{|\varphi(x)-\varphi(y)|:d(x,y)<\epsilon\}$.
Then,
$$P^{u}_{\mu}(f,\varphi,\epsilon,\eta(x),\gamma)\leq h^u_{\mu}(f)+\int_{M}\varphi d\mu+2\rho+\tau_{\epsilon}.$$
Let $\gamma\rightarrow0$ and $\epsilon\rightarrow0$ (hence $\tau_{\epsilon}\rightarrow0$), since $\rho>0$ is arbitrary, we obtain
$$P^{u}_{\mu}(f,\varphi)\leq h^u_{\mu}(f)+\int_{M}\varphi d\mu.$$

\end{proof}

\begin{lem}\label{metric pressure formular 2}
For any $\varphi \in C(M,\mathbb{R})$ and $\mu \in \mathcal{M}^{e}_{f}(M)$, we have
$$P_{\mu}^{u}(f, \varphi)\geq h^{u}_{\mu}(f)+\int_{M}\varphi d\mu.$$
\end{lem}

\begin{proof} For any $\epsilon>0$, any natural number $n$, any $\eta\in\mathcal{P}^u$, and any $0<\gamma<1$, we first give a lower bound for the minimal cardinality $\Sc$ of $(n, \epsilon, \gamma)\;u$-spanning sets of $\eta(x)$. 

Let us recall some facts
about the Hamming metric.
For positive integers $N$ and $n$, let us set
$$\omega_{N,n}=\left\{\omega=(\omega_{0},\cdots\cdots,\omega_{n-1})\mid \omega_{i}\in\{1,\cdots\cdots,N\}, 0\leq i \leq n-1\right\}.$$
The Hamming metric $\rho^{H}_{N,n}$ on $\omega_{N,n}$ is defined by $$\rho^{H}_{N,n}(\omega,\bar{\omega})=\dfrac{1}{n}\sum^{n-1}_{i=0}(1-\delta_{\omega_{i}\bar{\omega_{i}}}),$$
where $\delta_{kl}$ is the Kronecker symbol.

For $\omega\in\omega_{N,n}$, $r>0$, we denote by $B^{H}(\omega,r)$ the closed $r$-ball in the metric $\rho^{H}_{N,n}$  with
the center at $\omega$. The standard combinatorial arguments show that the number of points in $B^{H}(\omega,r)$, say $B(r, N, n),$ depends only on $r$, $N$, $n$ (not on $\omega$), and equals
$$B(r,N,n)=\sum^{[nr]}_{m=0}(N-1)^{m}C^{m}_{n}.$$
By the Stirling's formula, if $0< r< (N-1)/N$, then it is easy to see that
\begin{equation}\label{Stirling's formula}
\lim\limits_{n\rightarrow \infty}\dfrac{\log B(r,N,n)}{n}=r\log(N-1)-r\log r-(1-r)\log(1-r).
\end{equation} 

For any $y\in M$ and $\mathcal{P}_{\epsilon_{0}}\ni\alpha=\{A_1, \cdots\cdots, A_N\}$, set
$$\omega_{y,n}=\left\{\omega=(\omega_{0},\cdots\cdots,\omega_{n-1})\mid\omega_{i}\in\{1,\cdots\cdots,N\}, f^{i}y\in A_{\omega_{i}}, 0\leq i\leq n-1\right\}.$$
Now we define a semi-metric $d^{\alpha}_{n}$ on $M$ by
$$d^{\alpha}_{n}(y,z):=\rho^{H}_{N,n}(\omega_{y,n},\omega_{z,n})=\dfrac{1}{n}\sum^{n-1}_{i=0}(1-\delta_{\omega_{i}\bar{\omega_{i}}}).$$

Now for every $\epsilon>0$, 
set
$$\partial_{\epsilon}(\alpha)=\bigcup\limits_{A\in \alpha}\partial_{\epsilon}(A),$$ where $$\partial_{\epsilon}(A)=\{y\in A:\,\text{there exists a} \;z\in M\setminus A\;\text{such that}\,d(y,z)<\epsilon\}.$$ Since $\bigcap\limits_{\epsilon>0}\partial_{\epsilon}(\alpha)=\partial \alpha$, one has $\lim\limits_{\epsilon\to 0}\mu(\partial_{\epsilon}(\alpha))=\mu(\partial\alpha)$. (Moreover, we can assume that the measure $\mu$ is everywhere dense in $M$, $i.e.$, the measure
of any non-empty open subset of $M$ is positive.)

Let us focus on those partition $\alpha\in\mathcal{P}_{\epsilon_{0}}$ with $\mu(\partial\alpha)=0$. For any $s>0$, if $\epsilon$ is small enough, then $\mu(\partial_{\epsilon}(\alpha))<s^{2}/4$. 
If $y, z\in M$ and $d_{n}(y,z)<\epsilon$,
then for every $0\leq i\leq n-1$ either $f^{i} y$ and $f^{i} z$ belong to the same member
of $\alpha$, or both of them belong to $\partial_{\epsilon}(\alpha)$.
Let us denote for brevity the characteristic
function on $\partial_{\epsilon}(\alpha)$ by $\chi_{\epsilon}$ and set
$$B_{n, s}=\{y\in M\mid\sum^{n-1}_{i=0}\chi_{\epsilon}(f^{i} y)<\dfrac{ns}{2}\}.$$
Since $\int_{M}\chi_{\epsilon}d\mu<s^{2}/4$ and $f$ preserves the measure $\mu$, we have
\begin{align}\label{(10)}
\dfrac{ns^{2}}{4}&\geq\int_{M}\sum^{n-1}_{i=0}\chi_{\epsilon}(f^{i}y)d\mu\notag\\
&\geq\int_{M\setminus B_{n, s}}\sum^{n-1}_{i=0}\chi_{\epsilon}(f^{i}y)d\mu\notag\\
&\geq\dfrac{ns}{2}\mu\left(M\setminus B_{n,s}\right),
\end{align}
and so $\mu(B_{n,s})>1-s/2$. Hence for $\mu-a.e.\,x\in M$, one has $$\mu^{\eta}_x(B_{n, s}\cap \eta(x))>1-s/2.\quad (*)$$
If $y\in B_{n, s}$ and $d_{n}(y,z)<\epsilon$, then $d_{n}^{\alpha}(y, z)<s/2$. In other words, any intersection of an $\epsilon$-ball
in the metric $d_{n}$ with the set $B_{n, s}$ is contained in some $\epsilon/2$-ball in the semi-metric $d_{n}^{\alpha}(y, z)$. 

Since $\mu$ is ergodic, according to Theorem B in \cite{Hu1}, one has
 $$\lim_{n \to \infty}-\dfrac{1}{n}\log\mu_{x}^{\eta}(\alpha^{n-1}_{0}(x))= h^u_{\mu}(f), \;\mu-a.e.\,x.$$ Then for $\mu-a.e.\,x \in M$, one has $$\lim\limits_{n\to\infty}-\frac{1}{n}\log\mu^{\eta}_{x}(\alpha^{n-1}_{0}(y))=h^u_{\mu}(f), \;\;\;\mu^{\eta}_x-a.e.\,y\in\eta(x)$$ since $\mu^{\eta}_x=\mu^{\eta}_y$. Therefore, for $\mu-a.e.\;y$, there exists a $N(y)=N(y,\rho)>0$ such that if $n\geq N(y)$, then
$$\mu^{\eta}_{x}(\alpha^{n-1}_{0}(y))\leq e^{-n(h^u_{\mu}(f|\eta)-\rho)}.$$
Denote by $E_n=E_n(\rho)=\left\{y\in \eta(x)\mid N(y,\rho)\leq n\right\}$, then $E_n\subseteq E_{n+1}$ and $\mu(\cup_{n=1}^{\infty}E_n)=1$. So for each $\gamma>0$, there exists an $N$, such that $\mu^{\eta}_{x}(E_N)\geq1-\gamma$.

Now for each $x\in E_N$ with $(*)$ being true, consider a system $\mathfrak{U}$ of $S$ $\rho$-balls in the $d^{u}_{n}$-metric, such that these balls cover a subset $F_{n}\subseteq\eta(x)$ with $\mu^{\eta}_{x}(F_{n})\geq 1-\gamma$ (note that $S=\Sc$). In other words,
$$
\mathfrak{U}:=\left\{B^{u}_{n}(y_{i},\epsilon), 1\leq i\leq S\mid F_{n}\subseteq\bigcup\limits^{S}_{i=1}B^{u}_{n}(y_{i},\epsilon)
\,\text{and}\,\mu^{\eta}_{x}(F_{n})\geq 1-\gamma \right\}.$$
Then $$\mu^{\eta}_{x}(F_{n}\cap B_{n, s})\geq1-\gamma-s/2.$$
Suppose that $s<1-\gamma$, then $\mu^{\eta}_{x}(F_{n}\cap B_{(n, s)})>(1-\gamma)/2$. Since every ball $B^u_{y_i, \epsilon}$ is contained in $B_{n}(y_i, \epsilon)$, we claim that the intersection of every ball of $\mathfrak{U}$ with $B_{n, s}$ is contained in some $s/2$-ball in $d_{n}^{\alpha,u}$. Then there exist $S^{u}(f, \rho, \eta(x), \delta, \gamma)$ balls of radius $s/2$ in the metric $d_{n}^{\alpha,u}$, which cover the set $F_n\cap B_{n, s}$ whose $\mu^{\eta}_x$-measure is greater than $(1-\gamma)/2$.

To be precise, set $$P(n,y):=\left(\alpha(y),\alpha(fy),\alpha(f^{2}y),\cdots,\alpha(f^{n-1}y)\right),$$
we call $P(n,y)$ the $(\alpha,n)$-path of $y$. Suppose $V\in\alpha^{n-1}_{0}$, it is  obvious that for any two points $y, z\in V$, $P(n,y)=P(n,z)$, denote it by $P(n,V)$. Set $$B^{\mathfrak{U}}_{\frac{s}{2}}(y_{i}):=\{V\in \alpha^{n-1}_{0}\mid d_{n}^{\alpha,u}(P(n,V),P(n,y_{i}))<\frac{s}{2}\},$$
where $y_{i},\,i=1,2,\cdots,S^{u}(f, \rho, n, \eta(x), \gamma)$ are the centers of the balls in $\mathfrak{U}$. These are the $s/2$-balls we claimed. 

While for sufficiently large $n$,  some subset of the set $F_{n}\cap B_{n, s}$ with measure greater than $(1-\gamma)/4$ consists of elements of $\alpha^{n-1}_{0}\cap\eta(x)$ and the measure of such an element is less than
$e^{-n(h^u_{\mu}(f)-\rho)}$ by the conclusion before.
Consequently, the number of such elements is more than $(1-\gamma)e^{n(h^u_{\mu}(f)-\rho)}/4.$

Set $$B^{\mathfrak{U}}_{\frac{s}{2}}=\bigcup\limits^{S^{u}(f, \rho, n, \eta(x), \gamma)}_{i=1}B^{\mathfrak{U}}_{\frac{s}{2}}(y_{i}),$$ note that cardinality of each $B_{\frac{s}{2}}^{\mathfrak{U}}(y_i)$ is at most $B(\frac{s}{2}, |\alpha|,n)$, then
$$\text{Card}(B^{\mathfrak{U}}_{\frac{\epsilon}{2}})\leq S^{u}(f, \rho, n, \eta(x), \gamma)\cdot B(\frac{s}{2}, |\alpha|,n).$$
Thus we have
$$S^{u}(f, \rho, n, \eta(x), \gamma)\cdot B(\frac{s}{2}, |\alpha|,n)\geq\dfrac{(1-\gamma)e^{n(h^u_{\mu}(f)-\rho)}}{4}.$$

On the other hand,
since $\mu$ is ergodic, according to the Birkhoff's ergodic theorem, one has
 $$\lim_{n\rightarrow\infty}\dfrac{1}{n}S_{n}\varphi(y)=\int_{M}\varphi d\mu, \;\mu-a.e.\,y.$$
Hence for any $\lambda>0$ and $\mu-a.e.\;y$, there exists a $N(y)=N(y,\lambda)>0$ such that if $n\geq N(y)$, then
$$\dfrac{1}{n}S_{n}\varphi(y)\geq\int_{M}\varphi d\mu-\lambda.$$

 Set $H_n=H_n(\lambda)=\{y\in M\mid N(y,\lambda)\leq n \}$, then $H_n\subseteq H_{n+1}$ and $\mu(\cup_{n=1}^{\infty}H_n)=1$. So there exists an $N>0$ large enough such that $\mu(H_N)>1-\gamma/2$. Let $A_x$ be a subset of $\eta(x)$ with $\mu^{\eta}_x(A_x)>1-\gamma/2$ and $F'$ be an $(n, \epsilon)\;u$-spanning set of $A_x$ with cardinality $\Sc$. Set $A=A_x\cap H_N$, then $\mu^{\eta}_x(A)>1-\gamma$. Let $F\subseteq F'$ be an $(n, \epsilon)\;u$-spanning set of $A$ with smallest cardinality. Then for any $z\in F$, there exists $y(z)\in A$ such that $d^u_n(z, y(z))<\epsilon$.

Therefore,
$$ \begin{aligned}
\sum_{z\in F}\exp(S_{n}\varphi)(z)&\geq \sum_{z\in F}\exp((S_{n}\varphi)(y(z))-n\tau_{\epsilon})\notag\\
&\geq \sum_{z\in F}\exp(n(\int_{M}\varphi d\mu-\lambda)-n\tau_{\epsilon})\notag\\
&\geq \dfrac{(1-\gamma)e^{n(h^u_{\mu}(f)-\rho)}}{4B(\frac{s}{2}, |\alpha|, n)}\exp(n(\int_{M}\varphi d\mu-\lambda)-n\tau_{\epsilon}),
\end{aligned}$$
where $\tau_{\epsilon}:=\sup\{|\varphi(x)-\varphi(y)|:d(x,y)<\epsilon\}$.
Therefore,
$$P^{u}_{\mu}(f,\varphi,\rho,\eta(x),\gamma)\geq h^u_{\mu}(f)+\int_{M}\varphi d\mu-\lambda-\rho-\tau_{\epsilon}-O(s),$$ where $O(s)=\frac{s}{2}\log(N-1)-\frac{s}{2}\log\frac{s}{2}-(1-\frac{s}{2})\log(1-\frac{s}{2})$.
Since $\lambda, \rho, s, \epsilon$ are arbitrarily small, let them tend to $0$ (and hence $\tau_{\epsilon}\rightarrow0$ and $O(s)\to0$), we obtain
$$P^{u}_{\mu}(f,\varphi)\geq h^u_{\mu}(f,\alpha|\eta)+\int_{M}\varphi d\mu.$$

\end{proof}

\subsection{The case of sub-additive potentials--a proof of Theorem \ref{variational principle} }

\begin{lem}\label{lemma 2.10} 
Let $f : M \rightarrow M$ be a $C^{1}$-smooth partially hyperbolic diffeomorphism and $\mathcal{G} = \{\log g_{n}\}
_{n=1}^{\infty}$ be a sequence of sub-additive potentials of $f$. 
For any positive integer $l$ and small number $\rho>0$, there exists an $\epsilon_{0}>0$ such that for any $0<\epsilon<\epsilon_{0}$, the following inequality holds:
$$\sup_{z\in B^{u}_{n}(y,\epsilon)}\log g_{n}(z)\leq\sum^{n-1}_{i=0}\frac{1}{l}\log g_{l}(f^{i}y)+n\rho+C, \forall n, \forall y\in M,$$
where $B^{u}_{n}(y,\epsilon) = \{z \in W^{u}(x) : d^{u}_{n}(y, z) \leq\epsilon\}$ is the $(n, \epsilon)\; u$-Bowen
ball around $y$ and $C$ is a constant independent of $\rho$ and $\epsilon$.
\end{lem}
\begin{proof}
Note that the distance $d^u$ on the unstable manifold is equivalent to the Riemannian metric $d$ (see the observation in front of Proposition 2.4 of \cite{Zhu2} ), so any unstable local neighborhood $\overline{W^u(x, \delta)}$ is compact under $d^u$. Then one can get the desired result using a similar argument of Lemma 2.2 of \cite{Hu}.

\end{proof}

\textbf{Now we proceed to prove Theorem \ref{variational principle}.}

\begin{proof}

 First we prove $h^{u}_{\mu}(f)+\mathcal{G}_{*}(\mu)\geq P_{\mu}^{u}(f, \mathcal{G})$.

For any positive integer $l$ and any $\rho>0$,  by Lemma \ref{lemma 2.10}, there is a constant $C$ such that if $\epsilon$ is small enough, one has
$$
\begin{aligned}
&P_{\mu}^{u}(f, \mathcal{G}, \epsilon, n, \eta(x), \gamma)\\
=&\inf\left\{\sum_{y\in F}\sup_{z\in B_{n}^{u}(y,\epsilon)}g_{n}(z)\mid F\;\text{is an} \;(n, \epsilon, \gamma)\, u\text {-spanning subset of } \eta(x)\right\}\\
=&\inf\left\{\sum_{y\in F}\exp(\sup_{z\in B_{n}^{u}(y,\epsilon)}\log g_{n}(z))\mid F\;\text{is an} \;(n, \epsilon, \gamma)\, u\text {-spanning subset of } \eta(x)\right\}\\
\leq& e^{C+n\rho}\inf\left\{\sum_{y\in F}\exp(\frac{1}{l}\sum^{n-1}_{i=1}\log g_{l}(f^{i}(y)))\mid F\;\text{is an} \;(n, \epsilon, \gamma)\, u\text {-spanning subset of } \eta(x)\right\}.
\end{aligned}
$$
 
 Set $$M(n, \epsilon)=\inf\left\{\sum_{y\in F}\exp(\frac{1}{l}\sum^{n-1}_{i=1}\log g_{l}(f^{i}(y)))\mid F\;\text{is an} \;(n, \epsilon, \gamma)\, u\text {-spanning subset of } \eta(x)\right\}
,$$ then apply Theorem \ref{metric pressure formular} for the potential $\varphi=\dfrac{1}{l}\log g_l$, one has
$$
\begin{aligned}
&\lim_{\epsilon\rightarrow 0}\limsup\limits_{n\rightarrow\infty} \frac{1}{n}\log M(n, \epsilon)=&h^{u}_{\mu}(f)+\int_M\frac{1}{l}\log g_{l}d\mu.
\end{aligned}
$$
Therefore,
$$P_{\mu}^{u}(f, \mathcal{G}, \eta(x), \gamma)\leq h^{u}_{\mu}(f)+\int_M\frac{1}{l}\log g_{l}d\mu+\rho.$$
Let $l\rightarrow \infty$ and by the arbitrariness of $\rho$, one has
$$P_{\mu}^{u}(f, \mathcal{G})\leq  h^{u}_{\mu}(f)+\mathcal{G}_{*}(\mu).$$

Second, we prove the inverse inequality $$P_{\mu}^{u}(f, \mathcal{G})\geq  h^{u}_{\mu}(f)+\mathcal{G}_{*}(\mu).$$

For each $s>0$, there exists $0<\rho\leq s$, a measurable partition $\mathcal{P}\ni\alpha= \{A_{1}, \cdots\cdots, A_{m}\}$, and a finite open cover $\mathcal{U}=\{U_{1}, \cdots\cdots, U_{k}\}$ of $M$ with $k\geq m$, such that the following properties hold (using regularity of the measure $\mu$):
\begin{enumerate}
  \item $\mbox{diam}\,\alpha:= \sup\{\mbox{diam}\,A_{i} \mid A_{i} \in\alpha\}\leq s$ and $\mbox{diam}\,\mathcal{U}:= \sup\{\mbox{diam}\,U_{j} \mid U_{j} \in\mathcal{U}\}\leq s;$
  \item $\overline{U_{i}}\subseteq A_{i},\;1\leq i\leq m$;
  \item $\mu(A_{i}\setminus U_{i})\leq \rho,\;1\leq i\leq m$ and $\mu(\bigcup^{k}\limits_{i=m+1}U_{i})\leq \rho$;
  \item $2\rho\log m\leq s$.
\end{enumerate}

Set $$S_{n}(x):=\mbox{Card}\{0\leq l\leq n-1 \mid f^{l}(x)\in \bigcup^{k}_{i=m+1}U_{i}\},$$
We claim that there exists a $E_{N}\subseteq M$ with $\mu(E_{N})>1-\gamma$ such that  if $n\geq N$, then for any $y\in E_{N}$ one has
                      \begin{enumerate}
                         \item $S_{n}(y)\leq 2\rho n$;
                         \item $\mu^{\eta}_{y}(\alpha^{n-1}_{0}(y))\leq e^{-n(h^u_{\mu}(f|\eta)-\rho)}$;
                         \item $\mathcal{G}_{*}(\mu)-\rho\leq \dfrac{1}{n}\log g_n(y)\leq \mathcal{G}_{*}(\mu)+\rho$.
                       \end{enumerate}

Indeed since $\mu$ is ergodic, 
take $h$ to be the characteristic function on the set $\bigcup^{k}\limits_{i=m+1}U_{i}$, then $S_{n}(x)=\sum^{n-1}\limits_{i=0}h(f^{i}x)$. According to
the Birkhoff's ergodic theorem, one has
$$\lim\limits_{n\to \infty}\frac{1}{n}\sum^{n-1}\limits_{i=0}h(f^{i}y)=\int_M h d\mu=\mu(\bigcup^{k}_{i=m+1}U_i)\leq \rho, \;\mu-a.e.\,y.$$
 By the Sub-additive Ergodic Theorem, one has
$$\lim\limits_{n\to\infty}\dfrac{1}{n}\log g_n(y)= \mathcal{G}_{*}(\mu), \;\mu-a.e.\, y.$$
By Theorem B in \cite{Hu1}, one has 
 $$\lim_{n \to \infty}-\dfrac{1}{n}\log\mu_{y}^{\eta}(\alpha^{n-1}_{0}(y))= h^u_{\mu}(f|\eta), \;\mu-a.e.\,y.$$
 Hence, for $\mu-a.e.\,y$, there exists an $N(y)=N(y, \rho)>0$ such that if $n\geq N(y)$, then
 $$S_{n}(y)\leq 2n\rho, \;\mu^{\eta}_{x}(\alpha^{n-1}_{0}(y))\leq e^{-n(h^u_{\mu}(f|\eta)-\rho)},$$ and $$\mathcal{G}_{*}(\mu)-\rho\leq \dfrac{1}{n}\log g_n(y)\leq \mathcal{G}_{*}(\mu)+\rho.$$

 Set $E_n=\left\{y\in M\mid N(y)=N(y, \rho)\leq n\right\}$, then $\mu(\cup_{n=1}^{\infty}E_n)=1$. So there exists an $N>0$ large enough with $\mu(E_N)>1-\gamma$, such that if $n>N$,  then for any $y\in E_N$, one has
 \begin{enumerate}
 \item $S_{n}(y)\leq 2\rho n$;
 \item $\mu^{\eta}_{y}(\alpha^{n-1}_{0}(y))\leq e^{-n(h^u_{\mu}(f|\eta)-\rho)}$;
 \item $\mathcal{G}_{*}(\mu)-\rho\leq \dfrac{1}{n}\log g_n(y)\leq \mathcal{G}_{*}(\mu)+\rho$.
 \end{enumerate}
 Hence the claim above is verified.

 For the set $E_{N}$, there exists $x\in M$ such that $\mu^{\eta}_{x}(E_N)=\mu^{\eta}_{x}(E_{N}\cap \eta(x))>1-\gamma$. Set $A:=E_{N}\cap \eta(x)$, if $n>N$, then for every $y\in A$, one has \begin{enumerate}
 \item $S^{u}_{n}(y)\leq 2\rho n$;
 \item $\mu^{\eta}_{x}(\alpha^{n-1}_{0}(y))\leq e^{-n(h^u_{\mu}(f|\eta)-\rho)}$;
 \item $\mathcal{G}_{*}(\mu)-\rho\leq \dfrac{1}{n}\log g_n(y)\leq \mathcal{G}_{*}(\mu)+\rho$;
 \end{enumerate}
 where $S^{u}_{n}(y):=\mbox{Card}\{0\leq l\leq n-1 \mid f^{l}(x)\in \bigcup^{k}_{i=m+1}(U_{i}\cap\eta(x))\}.$

   Set $$(\alpha^{n-1}_{0})^{*}:=\{D\in \alpha^{n-1}_{0}\mid D\cap A\neq\emptyset\}.$$
Then for any $n\geq N$, one has

\begin{equation}\label{(1)}
\mbox{Card}((\alpha^{n-1}_{0})^{*})\geq\sum\limits_{D\in(\alpha^{n-1}_{0})^{*}}\mu^{\eta}_{x}(D)e^{n(h^u_{\mu}(f|\eta)-\rho)}\geq \mu^{\eta}_{x}(A)e^{n(h^u_{\mu}(f|\eta)-\rho)}.
\end{equation}
On the other hand, choose $C>1$ satisfies $ d(y, z)\leq d^{u}(y, z)\leq C d(y, z)$ for any $y, z\in\overline{\eta(x)}$. Let $2C\epsilon$ be less than the Lebesgue number of the open cover $\mathcal{U}$. Let $F'$ be an $(n,\epsilon)\,u$-spanning set of $A$. Suppose $F\subseteq F'$ satisfies that for any $y \in F$, $\overline{B^{u}_{n}(y, \epsilon)}\cap A\neq\emptyset$. For each $y\in F$ and $B=\overline{B^{u}_{n}(y, \epsilon)}$, set $$p(B, \alpha^{n-1}_{0})=\mbox{Card}\{C\in \alpha^{n-1}_{0}\mid C\cap B\cap A\neq\emptyset\}.$$
We now estimate the number $p(B, \alpha^{n-1}_{0})$. Note that $\overline{B^{u}(f^{j}y, \epsilon)}\subseteq U^{u}_{i_{l}}=U_{i_{l}}\cap \eta(x)$ for some $U_{i_{l}}\in\mathcal{U}$. If $1\leq i_{l}\leq m$, then $f^{-l}U^{u}_{i_{l}}\subseteq f^{-l}A^{u}_{i_{l}}$, where $A^{u}_{i_{l}}=A_{i_{l}}\cap \eta(x)$. If $m+1\leq i_{l}\leq k$, then there are at most $m$ sets of the form
$f^{-l}A^{u}_{i_{l}}$ which have non-empty intersection with $f^{-l}U^{u}_{i_{l}}$. Since $S^{u}_{n}(y)\leq 2n\rho$, one has $p(B, \alpha^{n-1}_{0})\leq m^{2n\rho}$.
Then it follows that
\begin{align}\label{(2)}
\mbox{Card}((\alpha^{n-1}_{0})^{*})&\leq\sum_{z\in F}p\left(\overline{B^{u}_{n}(y, \epsilon)}, \alpha^{n-1}_{0}\right)\\
&\leq \mbox{Card}(F)m^{2n\rho}= \mbox{Card}(F)e^{2n\rho\log m}.
\end{align}
Hence $$\mbox{Card}(F)\geq \mu^{\eta}_x(A)e^{n(h^u_{\mu}(f|\eta)-\rho)-2n\rho\log m},$$ together with the fact that $\overline{B^{u}_{n}(z, \epsilon)}\cap A\neq\emptyset$ for each
$z\in F$, then
$$\begin{aligned}
&\sum_{z\in F'}\exp\left(\sup_{y\in B^{u}_{n}(z, \epsilon)}\log g_{n}(y)\right)\\
\geq&\sum_{z\in F}\exp\left(\sup_{y\in B^{u}_{n}(z, \epsilon)}\log g_{n}(y)\right)\\
\geq &\mbox{Card}(F)\exp\left(n(\mathcal{G}_{*}(\mu)-\rho)\right)\\
\geq&\mu^{\eta}_{x}(A)\exp\left(n(h^u_{\mu}(f|\eta)+\mathcal{G}_{*}(\mu))-2n\rho-2\rho\log m)\right).
\end{aligned}$$
This
leads to
$$\begin{aligned}
\dfrac{1}{n}P_{\mu}^{u}(f, \mathcal{G}, \epsilon, n, \eta(x), \gamma)\geq\dfrac{1}{n}\mu^{\eta}_{x}(A)+h^u_{\mu}(f|\eta)+\mathcal{G}_{*}(\mu)-2\rho-\dfrac{1}{n}2\rho\log m,
\end{aligned}$$
Let $n\to\infty$, since $s$ is arbitrary, $\rho\leq s$, and $2\rho\log m\leq s$, one has $$P_{\mu}^{u}(f, \mathcal{G})\geq h^u_{\mu}(f)+\mathcal{G}_{*}(\mu).$$

\end{proof}
\begin{rem} From the proof above, one can see that for any $\mu\in\mathcal{M}^e_f$ the quantity $P_{\mu}^{u}\left(f, \mathcal{G},\eta(x), \gamma \right)$ in Definition \ref{metric-pressure} actually doesn't depend on $\gamma$ and $\eta\in \mathcal{P}^u$ for $\mu-$a.e.\,$x$.
\end{rem}

\section{Other Definitions of unstable measure theoretic pressure.}

In this section, we investigate other definitions of unstable pressure, in terms of the Bowen's picture and the capacity picture.

Let $\mathcal{G}=\{\log g_n\}_n$ be a sequence of sub-additive potentials of $f$ on $M$.
Let $Z\subseteq M$ be an arbitrary subset, and $Z$ needn't to be compact or $f$-invariant. Take $\eta\in \mathcal{P}^u$. Take the $(n, \epsilon)$ $u$-Bowen ball around $x$:$$B^u_n(x, \epsilon)=\{y\in W^u(x) \mid d_n^u(x, y)\leq \epsilon\}.$$ For each open cover $\Gamma=\{B_{n_i}^u(x_i, \epsilon)\}_{i\in I}$ of $\Z$, set $n(\Gamma)=\min\{n_i \mid i\in I\}$.

\begin{defi}\label{Bowen-Top-Pressure}
For $s\in \mathbb{R}$, $\delta>0$, $N\in \mathbb{N}$, $\epsilon>0$, $x\in M$, and $Z\subseteq M$, set $$M^u(\mathcal{G}, s, N, \epsilon, Z, \overline{W^u(x, \delta)}):=\inf\limits_{\Gamma}\{\sum\limits_{i}{\exp(-sn_i+\sup\limits_{y\in B^u_{n_i}(x_i, \epsilon)}\log g_{n_i}(y))}\},$$ where $\Gamma$ runs over all countable open covers $\Gamma=\{B_{n_i}^u(x_i, \epsilon)\}_{i\in I}$ of $Z\cap \overline{W^u(x, \delta)}$ with $n(\Gamma)\geq N$.

Let $$m^u(\mathcal{G}, s, \epsilon, Z, \overline{W^u(x, \delta)}):=\lim\limits_{N\to\infty}M^u(\mathcal{G}, s, N, \epsilon, Z, \overline{W^u(x, \delta)}),$$
$$\begin{aligned}
P^u_{B}(f, \mathcal{G}, \epsilon, Z, \overline{W^u(x, \delta)})&:=\inf\{s\mid m^u(\mathcal{G}, s, \epsilon, Z, \overline{W^u(x, \delta)})=0\},\\
&:=\sup\{s\mid m^u(\mathcal{G}, s, \epsilon, Z, \overline{W^u(x, \delta)})=\infty\},\\
\end{aligned}$$ and
$$P^u_{B}(f, \mathcal{G}, Z, \overline{W^u(x, \delta)}):=\liminf \limits_{\epsilon\to 0}P^u_B(f, \mathcal{G}, \epsilon, Z, \overline{W^u(x, \delta)}),$$ then define $$P^u_B(f, \mathcal{G}, Z):=\lim\limits_{\delta\to 0}\sup\limits_{x\in M}P^u_B(f, \mathcal{G}, Z, \overline{W^u(x, \delta)}).$$
We call $P^u_B(f, \mathcal{G}, Z)$ the Bowen unstable topological pressure of $f$ on the subset $Z$ w.\ r.\ t. $\mathcal{G}$.
\end{defi}

\begin{rem}
1.\,As a matter of fact, in Definition \ref{Bowen-Top-Pressure}, we don't have to take the limit with respect to $\delta\to 0$. This can be seen by a simple modification of the proof of Proposition 3.1 in \cite{Zhang}. 

2.\,With the replacement of $\overline{W^u(x, \delta)}$ by $\overline{\eta(x)}$, all the quantities above make sense and then we can define the following metric pressure. This replacement also applies to the next two definitions.
\end{rem}

\begin{defi}\label{Bowen-Metric-Pressure}
For $\mu\in \mathcal{M}_f(M)$, $x\in M$, and the conditional measure $\mu^{\eta}_x$ (recall that $\mu=\int \mu^{\eta}_xd\mu(x)$), we define 
$$P^u_{B, \,\mu}(f, \mathcal{G}, \eta(x)):=\inf\{P^u_B(f, \mathcal{G}, Z, \overline{\eta(x)})\mid \mu^{\eta}_x(Z)=1\},$$
$$P^u_{B, \,\mu}(f, \mathcal{G}):=\sup\limits_{\eta\in P^u}\int_{M}P^u_{B,\, \mu}(f, \mathcal{G}, \eta(x))d\mu(x),$$ which is called the Bowen unstable metric pressure of $f$ w.\ r.\ t. $\mathcal{G}$. 

\end{defi}

\begin{defi}\label{Capacity-Top-Pressure}

Set $$\Lambda^u(\mathcal{G}, n, \epsilon, Z, \overline{W^u(x, \delta)}):=\inf\limits_{\Gamma}\{\sum\limits_{i}\sup\limits_{y\in B^u_{n_i}(x_i, \epsilon)}g_n(y)\},$$ where $\Gamma$ runs over all open covers $\Gamma=\{B_{n_i}^u(x_i, \epsilon)\}_{i\in I}$ of $\Z$ with $n_i=n$ for all $i$.

Then define $$\underline{CP}^u(f, \mathcal{G}, \epsilon, Z, \overline{W^u(x, \delta)}):=\liminf\limits_{n\to\infty}\frac{1}{n}\log \Lambda^u(\mathcal{G}, n, \epsilon, Z, \overline{W^u(x, \delta)}),$$
$$\overline{CP}^u(f, \mathcal{G}, \epsilon, Z, \overline{W^u(x, \delta)}):=\limsup\limits_{n\to\infty}\frac{1}{n}\log \Lambda^u(\mathcal{G}, n, \epsilon, Z, \overline{W^u(x, \delta)}),$$
$$\underline{CP}^u(f, \mathcal{G}, Z, \overline{W^u(x, \delta)}):=\liminf\limits_{\epsilon\to0}\underline{CP}^u(f, \mathcal{G}, \epsilon, Z, \overline{W^u(x, \delta)}),$$
$$\overline{CP}^u(f, \mathcal{G}, Z, \overline{W^u(x, \delta)}):=\liminf\limits_{\epsilon\to0}\overline{CP}^u(f, \mathcal{G}, \epsilon, Z, \overline{W^u(x, \delta)}).$$ Then the lower and upper capacity unstable topological pressure of $f$ on $Z$ w.\ r.\ t. $\mathcal{G}$ are defined by $$\underline{CP}^u(f, \mathcal{G}, Z):=\lim\limits_{\delta\to 0}\sup\limits_{x\in M}\underline{CP}^u(f, \mathcal{G}, Z, \overline{W^u(x, \delta)}),$$ and $$\overline{CP}^u(f, \mathcal{G}, Z):=\lim\limits_{\delta\to0}\sup\limits_{x\in M}\overline{CP}^u(f, \mathcal{G}, Z, \overline{W^u(x, \delta)}).$$

\end{defi}

\begin{defi}\label{Capacity-Metric-Pressure}

For $\mu\in \mathcal{M}_f(M)$, $x\in M$, and the conditional measure $\mu^{\eta}_x$ (recall that $\mu=\int \mu^{\eta}_xd\mu(x)$), we define
$$\underline{CP}^u_{\mu}(f, \mathcal{G}, \eta(x)):=\lim\limits_{\gamma\to0}\inf\{\underline{CP}^u(f, \mathcal{G}, Z, \overline{\eta(x)})\mid \mu^{\eta}_{x}(Z)\geq1-\gamma\},$$ and 
$$\underline{CP}^u_{\mu}(f, \mathcal{G}):=\sup\limits_{\eta\in P^u}\int_M\underline{CP}^u_{\mu}(f, \mathcal{G}, \eta(x))d\mu(x).$$ This is called the lower capacity metric pressure of $f$ w.\ r.\ t. $\mathcal{G}$, and similarly the upper capacity metric pressure can be defined.


\end{defi}

Next we collect some basic properties of these pressures.
\begin{prop}\label{basic property}
For pressures defined above, the following properties hold.

i) $P(f, \mathcal{G}, Z_1)\leq P(f, \mathcal{G}, Z_2)$ if $Z_1\subseteq Z_2$, where $P$ can be chosen to be $P^u_B$, $\underline{CP}^u$, or $\overline{CP}^u$.

ii) $P(f, \mathcal{G}, \bigcup\limits_{i}Z_i)=\sup\limits_{i}P(f, \mathcal{G}, Z_i)$ for a family $\{Z_i\}_i$ of subsets of $M$, where $P$ can be chosen to be $P^u_B$, $\underline{CP}^u$, or $\overline{CP}^u$.


iii) $P^u_B(f, \mathcal{G}, Z)\leq \underline{CP}^u(f, \mathcal{G}, Z)\leq \overline{CP}^u(f, \mathcal{G}, Z)$ for any subset $Z\subseteq M$.

iv) For any $\mu\in \mathcal{M}_{f}(M)$, one has $$P^u_{B, \mu}(f, \mathcal{G})\leq \underline{CP}^u_{\mu}(f, \mathcal{G})\leq \overline{CP}^u_{\mu}(f, \mathcal{G}).$$
\end{prop}

\begin{proof}

i), ii) follow from the definition. iii) can be proved by a quite similar argument as the proof of Theorem 1.4 (a) in \cite{Bar}. iv) follows immediately from iii).

\end{proof}
To prove Theorem \ref{Main Result}, we prove the following two lemmas and our proof are influenced by arguments in \cite{Hu}.

\begin{lem}\label{Capacity-Inequality}
 For any $\mu\in \mathcal{M}^e_{f}(M)$, one has $$\overline{CP}^u_{\mu}(f, \mathcal{G})\leq h^u_{\mu}(f)+\mathcal{G}_{*}(\mu).$$

\end{lem}
\begin{proof}
For any positive integer $k$, any $\epsilon>0$, and any small number $\rho>0$, take $\eta\in \mathcal{P}^u$, by Lemma 3.2 in \cite{Hu1} and the Birkhoff's ergodic theorem, one has
$$\lim\limits_{n\to \infty}-\frac{1}{n}\log \mu^{\eta}_{y}(B_n^u(y, \epsilon/2))=h^u_{\mu}(f|\eta)$$ for $\mu$-a.e.\,$y$, and $$\lim\limits_{n\to \infty}\frac{1}{n}\sum\limits_{i=0}^{n-1}\frac{1}{k}\log g_k(f^iy)=\int\frac{1}{k}\log g_kd\mu$$ for $\mu$-a.e.\,$y$.

Hence for $\mu$-a.e.\,$y$, there exists an $N(y, \rho, \epsilon)>0$ such that if $n\geq N(y, \rho, \epsilon)$, then $$\mu^{\eta}_{y}(B_n^u(y, \epsilon/2))\geq e^{-n(h^u_{\mu}(f|\eta)+\rho)},$$ and $$\left|\frac{1}{n}\sum\limits_{i=0}^{n-1}\frac{1}{k}\log g_k(f^iy)-\int\frac{1}{k}\log g_kd\mu\right|\leq \rho.$$ Set $K_n(\rho, \epsilon)=\{y\in M\mid N(y, \rho, \epsilon)\leq n\}$. Then $K_n(\rho, \epsilon)\subseteq K_{n+1}(\rho, \epsilon)$, and $\mu(\bigcup\limits_{n=1}^{\infty}K_n(\rho, \epsilon))=1$. So there exists an $N>0$ such that $\mu(K_N(\rho, \epsilon))>1-\rho$. Furthermore, for each $x\in K_N(\rho, \epsilon)$, let $G(x)= \overline{\eta(x)}\cap K_N(\rho, \epsilon)$, then $\mu^{\eta}_{x}(G(x))\geq 1-\rho$, and for each $y\in G(x)$, one has
 \begin{equation}\label{measure-inequality}
 \mu^{\eta}_{y}(B_n^u(y, \epsilon/2))=\mu^{\eta}_{x}(B_n^u(y, \epsilon/2))\geq e^{-n(h^u_{\mu}(f|\eta)+\rho)} \,(\text{since}\, \mu^{\eta}_{y}=\mu^{\eta}_{x}).
 \end{equation}

By Lemma \ref{lemma 2.10}, one has $$\sup\limits_{z\in B_n^u(y, \epsilon)}\log g_n(z)\leq n\int\frac{1}{k}\log g_kd\mu+2n\rho+C.$$  Let $E$ be an $(n, \epsilon)$ $u$-separated set of $K_{N}(\rho, \epsilon)$ with the largest cardinality. Then $$\overline{\eta(x)}\cap K_N(\rho, \epsilon)\subseteq \bigcup\limits_{y\in E}B_n^u(y, \epsilon).$$ Furthermore, the $u$-balls $\{B_n^u(y, \epsilon/2)\mid y\in E\}$ are mutually disjoint, and by (\ref{measure-inequality}), the cardinality of $E$ is less than or equal to $e^{n(h^u_{\mu}(f|\eta)+\rho)}$.

Therefore, $$\begin{aligned}
\Lambda^u(\mathcal{G}, n, \epsilon, K_N(\rho, \epsilon), \overline{\eta(x)})&\leq \sum\limits_{y\in E}\sup\limits_{z\in B_n^u(y, \epsilon)}g_n(z)\\
&\leq e^{n(h^u_{\mu}(f|\eta)+\rho)}e^{n(\int\frac{1}{k}\log g_kd\mu+2\rho)+C}
\end{aligned}$$ Hence $$\overline{CP}^u(f, \mathcal{G}, K_N(\rho, \epsilon), \overline{\eta(x)})\leq h^u_{\mu}(f|\eta)+\int\frac{1}{k}\log g_kd\mu+3\rho,$$ and so $$\overline{CP}^u(f, \mathcal{G}, \eta(x))\leq h^u_{\mu}(f|\eta)+\int\frac{1}{k}\log g_kd\mu+3\rho.$$ Let $k\to \infty$, by the arbitrariness of $\rho$ and Theorem A in \cite{Hu1}, one gets that $$\overline{CP}^u_{\mu}(f, \mathcal{G}, \eta(x))\leq h^u_{\mu}(f)+\mathcal{G}_{*}(\mu).$$ Therefore, $$\overline{CP}^u_{\mu}(f, \mathcal{G})\leq h^u_{\mu}(f)+\mathcal{G}_{*}(\mu).$$

\end{proof}

\begin{lem}\label{Bowen-Inequality}
 For any $\mu\in \mathcal{M}^e_{f}(M)$, one has $$P^u_{B, \mu}(f, \mathcal{G})\geq h^u_{\mu}(f)+\mathcal{G}_{*}(\mu).$$

\end{lem}
\begin{proof}
For any $\rho>0$ and $\gamma\in(0, 1/2)$ and set $\lambda=h^u_{\mu}(f)+\mathcal{G}_{*}(\mu)-2\rho$. Take $\eta\in \mathcal{P}^u$, then for any $\epsilon>0$, by Lemma 3.2 in \cite{Hu1} and the sub-additive ergodic theorem, one has
$$\lim\limits_{n\to \infty}-\frac{1}{n}\log \mu^{\eta}_{y}(B_n^u(y, \epsilon))=h^u_{\mu}(f|\eta),\;\text{and}\;
\lim\limits_{n\to \infty}\frac{1}{n}\log g_n(y)=\mathcal{G}_{*}(\mu)$$ for $\mu$-a.e.\,$y$.

Hence for $\mu$-a.e.\,$y$, there exists an $N(y, \rho, \epsilon)>0$ such that if $n\geq N(y, \rho, \epsilon)$, then $$\mu^{\eta}_{y}(B_n^u(y, \epsilon))\leq e^{-n(h^u_{\mu}(f|\eta)-\rho)}\;\text{and}\; \frac{1}{n}\log g_n(y)\geq \mathcal{G}_{*}(\mu)-\rho.$$
Set $$K_n(\rho, \epsilon)=\{y\in M\mid N(y, \rho, \epsilon)\leq n\}.$$ Then $K_n(\rho, \epsilon)\subseteq K_{n+1}(\rho, \epsilon)$, and $\mu(\bigcup\limits_{n=1}^{\infty}K_n(\rho, \epsilon))=1$, then $$\mu^{\eta}_x(\overline{\eta(x)}\cap \bigcup\limits_{n=1}^{\infty}K_n(\rho, \epsilon))=1.$$ For any $Z\subseteq M$ with $\mu^{\eta}_x(Z)=1$, set $K'=Z\cap\overline{\eta(x)}\cap\bigcup\limits_{n=1}^\infty K_n(\rho, \epsilon)$, Then $\mu^{\eta}_x(K')=1$. So there exists an $N>0$ such that $\mu^{\eta}_x(K_N(\rho, \epsilon)\cap \overline{\eta(x)}\cap Z)>1-\gamma$. Set $G_N(x)= \overline{\eta(x)}\cap K_N(\rho, \epsilon)\cap Z$, then $\mu^{\eta}_{x}(G_N(x))\geq 1-\gamma$, and for each $y\in G_N(x)$, one has
\begin{equation}\label{measure-inequality 1}
\mu^{\eta}_{x}(B_n^u(y, \epsilon))\leq e^{-n(h^u_{\mu}(f|\eta)-\rho)} \,(\text{since}\, \mu^{\eta}_{y}=\mu^{\eta}_{x}).
\end{equation}

Take any countable open cover $\Gamma=\{B_{n_i}^u(y_i, \epsilon/2)\}_i$ of $G_N(x)$ with $n(\Gamma)\geq N$. We can assume $G_N(x)$ is compact, otherwise approximate it by a compact subset within an error. Then we may assume this cover is finite, say $\{B^u_{n_1}(y_1, \epsilon/2), ..., B^u_{n_l}(y_l, \epsilon/2)\}$. For each $i=1,..., l$, we can choose $z_i\in G_N(x)\cap B_{n_i}^u(y_i, \epsilon/2)$, then $B_{n_i}^u(y_i, \epsilon/2)\subseteq B_{n_i}^u(z_i, \epsilon)$, and $\{B_{n_i}^u(z_i, \epsilon)\}_i$ forms an open cover of $G_N(x)$.
Then
$$\begin{aligned}
&\sum\limits_{B_{n_i}^u(z_i, \epsilon)\in \Gamma}\exp(-n_i\lambda+\sup\limits_{y\in B_{n_i}^u(z_i, \epsilon)}\log g_{n_i}(y))\\
&\geq\sum\limits_{i=1}^{l}\exp(-n_i\lambda+\log g_{n_i}(z_i))\\
&\geq \sum\limits_{i=1}^{l}\exp(-n_i\lambda+n_i(\mathcal{G}_{*}(\mu)-\rho))\\
&=\sum\limits_{i=1}^{l}\exp(-n_i(h^u_{\mu}(f)-\rho))\\
&\geq \sum\limits_{i=1}^{l}\mu^{\eta}_{x}(B_{n_i}^u(z_i, \epsilon))\\
&>1-\gamma>\frac{1}{2}.
\end{aligned}
$$
Hence $$M^u(\mathcal{G}, \lambda, n, \epsilon, K_N\cap Z, \overline{\eta(x)})>\frac{1}{2}\; \text{for any}\; \epsilon,$$ Thus $$m^u(\mathcal{G}, \lambda, \epsilon, K_N\cap Z, \overline{\eta(x)})>\frac{1}{2},$$  $$P^u_B(f, \mathcal{G}, \epsilon, K_N\cap Z, \overline{\eta(x)})\geq \lambda,$$ and $$P^u_B(f, \mathcal{G}, K_N\cap Z, \overline{\eta(x)})\geq \lambda.$$
Then by the arbitrariness of $\rho$, one has
$$
\begin{aligned}
P^u_B(f, \mathcal{G}, Z, \overline{\eta(x)})&\geq P^u_B(f, \mathcal{G}, K_N(\rho, \epsilon)\cap Z, \overline{\eta(x)})\\
&\geq h^u_{\mu}(f)+\mathcal{G}_{*}(\mu).
\end{aligned}
$$
Therefore, $$P^u_{B, \mu}(f, \mathcal{G})\geq h^u_{\mu}(f)+\mathcal{G}_{*}(\mu).$$

\end{proof}

\textbf{Now we proceed to prove Theorem \ref{Main Result} :}

\begin{proof}
It follows from Theorem \ref{variational principle}, Lemma \ref{Capacity-Inequality}, \ref{Bowen-Inequality}, and Proposition \ref{basic property}.
\end{proof}

\begin{coro}

 Given any $\mu\in \mathcal{M}^e_{f}(M)$, and any sequence of sub-additive potentials $\mathcal{G}=\{\log g_n\}_{n=1}^{\infty}$ of $f$ on $M$.
Set $$\begin{aligned}
K=\{x\in M\mid &\lim\limits_{\epsilon\to 0}\lim\limits_{n\to \infty}\frac{-1}{n}\log \mu^{\eta}_{x}(B^u_n(x, \epsilon))=h^u_{\mu}(f)\\      &\text{and}\,\lim\limits_{n\to \infty}\frac{1}{n}\log g_n(x)=\mathcal{G}_{*}(\mu)\},
\end{aligned}$$ 
then $$P^u_{\mu}(f, \mathcal{G})=P^u_{B}(f, \mathcal{G}, K)=\underline{CP}^u(f, \mathcal{G}, K)=\overline{CP}^u(f, \mathcal{G}, K).$$

\end{coro}
\begin{proof} It is easy to see that $\mu(K)=1$.
For any positive integer $k$, any $\epsilon>0$, and any small number $\rho>0$, take $\eta\in \mathcal{P}^u$, by Lemma 3.2 in \cite{Hu1} and the Birkhoff's ergodic theorem, one has
$$\lim\limits_{n\to \infty}-\frac{1}{n}\log \mu^{\eta}_{y}(B_n^u(y, \epsilon/2))=h^u_{\mu}(f|\eta)$$ for $\mu$-a.e. $y$, and $$\lim\limits_{n\to \infty}\frac{1}{n}\sum\limits_{i=0}^{n-1}\frac{1}{k}\log g_k(f^iy)=\int\frac{1}{k}\log g_kd\mu$$ for $\mu$-a.e. $y$.

For any $y\in K$, there exists an $N(y)=N(y, \rho, \epsilon)>0$ such that if $n\geq N(y)$, then
$$\left|\frac{1}{n}\log \mu^{\eta}_{y}(B_{n_i}^u(y, \epsilon/2))+h^u_{\mu}(f|\eta)\right|\leq \rho,$$
 and
 $$\left|\frac{1}{n}\sum\limits_{i=0}^{n-1}\frac{1}{k}\log g_k(f^iy)-\int\frac{1}{k}\log g_kd\mu\right|\leq \rho.$$ Set $K_n=\{y\in K\mid N(y, \rho, \epsilon)\leq n\}$, then $K=\bigcup_{n\geq 1}K_n$.
By a quite similar proof of Lemma \ref{Capacity-Inequality} and let $n\to \infty$, we get
$$\overline{CP}^u(f, \mathcal{G}, K, \overline{W^u(x, \delta)})\leq h^u_{\mu}(f|\eta)+\int\frac{1}{k}\log g_kd\mu+3\rho.$$
Let $k\to \infty$, by the arbitrariness of $\rho$ , one gets
$$\overline{CP}^u(f, \mathcal{G},K)\leq h^u_{\mu}(f)+\mathcal{G}_{*}(\mu).$$ 
On the other hand, by  Proposition \ref{basic property}, one has $$P^u_B(f, \mathcal{G}, K)\leq \underline{CP}^u(f, \mathcal{G},K)\leq\overline{CP}^u(f, \mathcal{G},K),$$ and by a similar proof of Lemma \ref{Bowen-Inequality}, one gets $$P^u_B(f, \mathcal{G}, Z, \overline{W^u(x, \delta)})\geq h^u_{\mu}(f)+\mathcal{G}_{*}(\mu).$$ for any unstable neighborhood $\overline{W^u(x, \delta)}$, hence $$P^u_B(f, \mathcal{G}, K)\geq h^u_{\mu}(f)+\mathcal{G}_{*}(\mu),$$ then the conclusion follows based on Theorem \ref{variational principle}.
\end{proof}

\proof[Acknowledgements] The first author is supported by a NSFC (National Science Foundation of China) grant with grant No.\,11501066 and a grant from the Department of Education in Chongqing City with contract No.\,KJ1705122 in Chongqing Jiaotong University; she is also supported by the Program of Chongqing Innovation Team Project in University under Grant CXTDX201601022 in Chongqing Jiaotong University.

The second author is supported by the Fundamental Research Funds for the Central Universities with Project No.\,2018CDXYST0024 in Chongqing University.

The third author is supported by the National Science Foundation of China with grant No.\,11871120; he is also supported by the Foundation and Frontier Research Program of Chongqing (cstc2016jcyjA0312) and the Fundamental Research Funds for the Central Universities with Project No.\,2018CDQYST0023.

\end{document}